\documentclass[11pt]{article}
\usepackage{mathrsfs}
\usepackage{amsfonts}
\usepackage{amsmath}
\usepackage{amssymb}
\usepackage{latexsym}
\usepackage{amsthm}
\usepackage{indentfirst}
\usepackage{graphicx}
\usepackage{cite}
\usepackage{bm}

\newcommand{\ball}{\bm{\mathrm{B}}}
\newcommand{\dif}{\mathrm{d}}
\newcommand{\I}{\bm{\mathrm{I}}}

\newtheorem{theorem}{Theorem}%[section]
\newtheorem{lemma}{Lemma}%[section]
\theoremstyle{definition}

\newtheorem{remark}{Remark}%[section]
\newtheorem{definition}{Definition}%[section]
\numberwithin{equation}{section}

\allowdisplaybreaks

\begin{document}

\title{Semilocal Convergence Behavior of Halley's Method Using
Kantorovich's Majorants Principle }
\author{Yonghui Ling\,$^\textup{a,}$\footnote{Corresponding author.
\newline \indent
    {\it E-mail:} lingyinghui@163.com (Y. Ling),
    xxu@zjnu.cn (X. Xu).}\ , Xiubin
Xu$^\textup{b,}$\footnote{The second author's work was supported in
part by the National Natural Science Foundation of China (Grant No.
61170109 and No. 10971194).}
\\\small\it
\textup{a} Department of Mathematics, Zhejiang University, Hangzhou
310027, China
\\\small\it
\textup{b} Department of Mathematics, Zhejiang Normal University,
Jinhua 321004, China }

\date{}
\maketitle

\begin{center}
\begin{minipage}{135mm}
\textbf{Abstract:} The present paper is concerned with the semilocal
convergence problems of Halley's method for solving nonlinear
operator equation in Banach space. Under some so-called majorant
conditions, a new semilocal convergence analysis for Halley's method
is presented. This analysis enables us to drop out the assumption of
existence of a second root for the majorizing function, but still
guarantee Q-cubic convergence rate. Moreover, a new error estimate
based on a directional derivative of the twice derivative of the
majorizing function is also obtained. This analysis also allows us
to obtain two important special cases about the convergence results
based on the premises of Kantorovich and Smale types.
\end{minipage}
\end{center}

\section{Introduction}
\label{section:Introduction}

In this paper, we concern with the numerical approximation of the
solution $x$ of the nonlinear equation
\begin{equation}
\label{eq:NonlinearOperatorEquation} F(x) = 0,
\end{equation}
where $F$ is a given nonlinear operator which maps from some
nonempty open convex subset $D$ in a Banach space $X$ to another
Banach space $Y$. Newton's method with initial point $x_0$ is
defined by
\begin{equation}
\label{iteration:NewtonMethod}
    x_{k+1} = x_k - F'(x_k)^{-1}F(x_k), \ \ \ k = 0,1,2,\ldots,
\end{equation}
which is the most efficient method known for solving such an
equation. One of the famous results on Newton's method
(\ref{iteration:NewtonMethod}) is the well-known Kantorovich theorem
\cite{Kantorvich1982}, which guarantees convergence of that method
to a solution using semilocal conditions. It does not require a
priori existence of a solution, proving instead the existence of the
solution and its uniqueness on some region. Another important result
concerning Newton's method (\ref{iteration:NewtonMethod}) is Smale's
point estimate theory \cite{Smale1986}. It assumes that the
nonlinear operator is analytic at the initial point.

Since then, Kantorovich like theorem has been the subject of many
new researches, see for example, \cite{GraggTapia1974,
Deuflhard1979, Ypma1982, Gutierrez2000, XuLi2007,XuLi2008}. For
Smale's point estimate theory, Wang and Han in \cite{WangHan1997}
discussed $\alpha$ criteria under some weak condition and
generalized this theory. In particular, Wang in \cite{Wang1999}
introduced some weak Lipschitz conditions called Lipschitz
conditions with L-average, under which Kantorovich like convergence
criteria and Smale's point estimate theory can be investigated
together.

Recently, Ferreira and Svaiter \cite{Ferreira2009a} presented a new
convergence analysis for Kantorovich's theorem which makes clear,
with respect to Newton's method (\ref{iteration:NewtonMethod}), the
relationship of the majorizing function $h$ and the nonlinear
operator $F$ under consideration. Specifically, they studied the
semilocal convergence of Newton's method
(\ref{iteration:NewtonMethod}) under the following majorant
conditions:
$$
\|F'(x_0)^{-1}[F'(y) - F'(x)]\| \leq h'(\|y - x\| + \|x - x_0\|) -
h'(\|x - x_0\|), \ \ x,y \in \ball(x_0,R), R > 0,
$$
where $\|y - x\| + \|x - x_0\| < R$ and $h:[0,R) \to \mathbb{R}$ is
a continuously differentiable, convex and strictly increasing
function and satisfies $h(0) > 0, h'(0) = - 1$, and has zero in
$(0,R)$. This convergence analysis relaxes the assumptions for
guaranteeing Q-quadratic convergence (see Definition
\ref{definition:Q-orderConvergence}) of Newton's method
(\ref{iteration:NewtonMethod}) and obtains a new estimate of the
Q-quadratic convergence. This analysis was also introduced in
\cite{Ferreira2009b} studing the local convergence of Newton's
method.

Halley's method in Banach space denoted by
\begin{equation}
\label{iteration:HalleyMethod} x_{k+1} = x_k - [\I -
L_F(x_k)]^{-1}F'(x_k)^{-1}F(x_k), \ \ k = 0,1,2,\ldots,
\end{equation}
where operator $L_F(x)=\frac{1}{2}F'(x)^{-1}F''(x)F'(x)^{-1}F(x)$,
is another famous iteration for solving nonlinear equation
(\ref{eq:NonlinearOperatorEquation}). The results concerning
convergence of this method with its modification have recently been
studied under the assumptions of Newton-Kantorovich type, see for
example, \cite{Candela1990a, HanWang1997, Argyros2004, YeLi2006,
Ezquerro2005, Argyros2012}. Besides, there are also some researches
concerned with Smale-type convergence for Halley's method
(\ref{iteration:HalleyMethod}), if the nonlinear operator $F$ is
analytic at the initial point, see for example,
\cite{WangHan1990,Wang1997,Han2001}.

Motivated by the ideas of Ferreira and Svaiter in
\cite{Ferreira2009a}, in the rest of this paper, we study the
semilocal convergence of Halley's method
(\ref{iteration:HalleyMethod}) under some so-called majorant
conditions.

Suppose that $F$ is a twice Fr\'{e}chet differentiable operator and
there exists $x_0 \in D$ such that $F'(x_0)$ is nonsingular. In
addition, let $R > 0$ and $h : [0,R) \to \mathbb{R}$ be a twice
continuously differentiable function. We say the operator $F''$
satisfies the majorant conditions, if
\begin{equation}
\label{condition:MajorantCondition} \|F'(x_0)^{-1}[F''(y) -
F''(x)]\| \leq h''(\|y - x\| + \|x - x_0\|) - h''(\|x - x_0\|), \ \
x,y \in \ball(x_0,R),
\end{equation}
where $\|y - x\| + \|x - x_0\| < R$ and the following assumptions
hold:
\begin{enumerate}
\item[\textup{(A1)}]
$h(0) > 0, h''(0) > 0, h'(0) = -1.$
\item[\textup{(A2)}]
$h''$ is convex and strictly increasing in $[0,R)$.
\item[\textup{(A3)}]
$h$ has zero(s) in $(0,R)$. Assume that $t^*$ is the smallest zero
and $h'(t^*) < 0.$
\end{enumerate}

Under the assumptions that the second derivative of $F$ satisfies
the majorant conditions, we establish a semilocal convergence for
Halley's method (\ref{iteration:HalleyMethod}). In our convergence
analysis, the assumptions for guaranteeing Q-cubic convergence of
Halley's method (\ref{iteration:HalleyMethod}) are relaxed. In
addition, we obtain a new error estimate based on a directional
twice derivative of the derivative of the majorizing function. We
drop out the assumption of existence of a second root for the
majorizing function, still guaranteeing Q-cubic convergence.
Moreover, the majorizing function even do not need to be defined
beyond its first root. In particular, this convergence analysis
allows us to obtain some important special cases, which includes
Kantorovich-type convergence results under Lipschitz conditions and
Smale-type convergence results under the $\gamma$-condition (see
Definition \ref{definition:GammaCondition}).

The rest of this paper is organized as follows. In Section 2, we
introduce some preliminary notions and properties of the majorizing
function. In Section 3, we study the majorizing function and the
results regarding only the majorizing sequence. The main results
about the semilocal convergence and new error estimate are stated
and proved in Section 4. In Section 5, we present two special cases
of our main results. And finally in Section 6, some remarks and
numerical example are offered.

\section{Preliminaries}
\label{section:Preliminaries}

Let $X$ and $Y$ be Banach spaces. For $x \in X$ and a positive
number $r$, throughout the whole paper, we use $\ball(x,r)$ to stand
for the open ball with radius $r$ and center $x$, and let
$\overline{\ball(x,r)}$ denote its closure.

Throughout this paper, for a convergent sequence $\{x_n\}$ in $X$,
we use the notion of Q-order of convergence (see \cite{Jay2001} or
\cite{Potra1989} for more details).

\begin{definition}
\label{definition:Q-orderConvergence} A sequence $\{x_n\}$ converges
to $x^*$ with Q-order (at least) $q \geq 1$ if there exist two
constants $c \geq 0$ and $N \geq 0$ such that for all $n \geq N$ we
have
$$
\|x^* - x_{n+1}\| \leq c\|x^* - x_n\|^q.
$$
For $q = 2,3$ the convergence is said to be (at least) Q-quadratic
and Q-cubic, respectively.
\end{definition}

The notions about Lipschitz condition (see
\cite{Wang2000,Deuflhard2004}) and the $\gamma$-condition (see
\cite{WangHan1997}) are defined as follows.

\begin{definition}[Lipschitz Condition]
\label{definition:LipschitzCondition} The condition on the operator
$F$
$$
\|F(x) - F(y)\| \leq L\|x - y\|,\ \ \ x,y\in D
$$
is usually called the Lipschitz condition in the domain $D$ with
constant $L$. If it is only required to satisfy
$$
\|F(x) - F(x_0)\| \leq L\|x - x_0\|,\ \ \ x \in \ball(x_0,r),
$$
we call it the center Lipschitz condition in the ball
$\ball(x_0,r)$. In particular, if $F'(x_0)^{-1}F'$ satisfies the
Lipschitz condition, i.e.
$$
\|F'(x_0)^{-1}[F'(x) - F'(y)]\| \leq L\|x - y\|,\ \ \ x,y \in
\ball(x_0,r),
$$
we call it the affine covariant Lipschitz condition. The
corresponding center Lipschitz condition is referred to as affine
covariant center Lipschitz condition.
\end{definition}

\begin{definition}[$\gamma$-Condition]
\label{definition:GammaCondition} Let $F:D\subset X \to Y$ be a
nonlinear operator with thrice continuously differentiable, $D$ open
and convex. Suppose $x_0 \in D$ is a given point, and let $0 < r
\leq 1/\gamma$ be such that $\ball(x_0,r) \subset D$. $F$ is said to
satisfy the $\gamma$-condition (with 1-order) on $\ball(x_0,r)$ if
$$
\|F'(x_0)^{-1}F''(x)\| \leq \frac{2\gamma}{(1 - \gamma\|x -
x_0\|)^3}.
$$
$F$ is said to satisfy the $\gamma$-condition with 2-order on
$\ball(x_0,r)$ if the following relation holds:
$$
\|F'(x_0)^{-1}F'''(x)\| \leq \frac{6\gamma^2}{(1 - \gamma\|x -
x_0\|)^4}.
$$
\end{definition}

For the convergence analysis, we need the following useful lemmas
about elementary convex analysis. The first one is slightly modified
from the one in \cite{Ferreira2009b}.
\begin{lemma}
\label{lemma:ConvexFunctionProperties1} Let $R > 0$. If $g : [0,R)
\to \mathbb{R}$ is continuously differentiable and convex, then
\begin{enumerate}
\item[\textup{(i)}]
$(1 - \theta)g'(\theta t) \leq \cfrac{g(t) - g(\theta t)}{t} \leq (1
- \theta)g'(t)$ for all $t \in (0,R)$ and $0 \leq \theta \leq 1$.
\item[\textup{(ii)}]
$\cfrac{g(u) - g(\theta u)}{u} \leq \cfrac{g(v) - g(\theta v)}{v}$
for all $u,v \in [0,R),\ u < v$ and $0 \leq \theta \leq 1$.
\end{enumerate}
\end{lemma}

\begin{lemma}[\cite{Ferreira2009a}]
\label{lemma:ConvexFunctionProperties2} Let $I \subset \mathbb{R}$
be an interval and $g : I \to \mathbb{R}$ be convex. Then
\begin{enumerate}
\item[\textup{(i)}]
For any $u_0 \in \textup{int}(I)$, there exists $(in \ \mathbb{R})$
\begin{equation}
\label{definition:DirectionalDerivative}
 D^-g(u_0) := \lim_{u \to u_0^-} \frac{g(u_0) - g(u)}{u_0 - u} =
\sup_{u < u_0} \frac{g(u_0) - g(u)}{u_0 - u}.
\end{equation}
\item[\textup{(ii)}]
If $u, v, w \in I$ and $u \leq v \leq w$, then
$$
g(v) - g(u) \leq [g(w) - g(u)]\frac{v - u}{w - u}.
$$
\end{enumerate}
\end{lemma}

For the convenience of analysis, we define the majorizing function
with respect to Halley's method (\ref{iteration:HalleyMethod}) as
follows.

\begin{definition}
\label{definition:MajorizingFunction} Let $F: D \subset X \to Y$ be
a twice continuously differentiable nonlinear operator. For a given
guess $x_0 \in D$, we assume $F'(x_0)$ is nonsingular. A
continuously twice differentiable function $h: [0,R) \to \mathbb{R}$
is said to be a majorizing function to $F$ at $x_0$, if $F''$
satisfies the majorant conditions in $\ball(x_0,R) \subset D$ and
the following initial conditions:
\begin{equation}
\label{condition:InitialCondition} \|F'(x_0)^{-1}F(x_0)\| \leq
h(0),\ \|F'(x_0)^{-1}F''(x_0)\| \leq h''(0).
\end{equation}
\end{definition}

The following lemma describes some basic properties of the
majorizing function $h$.

\begin{lemma}
\label{lemma:MajorizingFunctionProperties} Let $R > 0$ and let $h :
[0,R) \to \mathbb{R}$ be a twice continuously differentiable
function which satisfies assumptions $(A1)-(A3)$. Then
\begin{enumerate}
\item[\textup{(i)}]
$h'$ is strictly convex and strictly increasing on $[0,R)$.
\item[\textup{(ii)}]
$h$ is strictly convex on $[0,R)$, $h(t) > 0$ for $t \in [0,t^*)$
and equation $h(t) = 0$ has at most one root on $(t^*,R)$.
\item[\textup{(iii)}]
$-1 < h'(t) < 0$ for $t \in (0,t^*)$.
\end{enumerate}
\end{lemma}

\begin{proof}
(i) follows from assumption (A2) and $h''(0) > 0$ in (A1). (i)
implies that $h$ is strictly convex. As assumption (A1), (i) and
$h(t^*) = 0$, we know that $h(t) = 0$ has at most one root on
$(t^*,R)$. Since $h(t^*) = 0$ and $h(0) > 0$, one has $h(t) > 0$ for
$t \in [0,t^*)$. It remains to show (iii). Firstly, since $h$ is
strictly convex, we obtain from Lemma
\ref{lemma:ConvexFunctionProperties1} that
$$
h'(t) < \frac{h(t^*) - h(t)}{t^* - t},\ \ \ t \in [0,t^*).
$$
This implies $0 = h(t^*) > h(t) + h'(t)(t^* - t)$. In view of $h(t)
> 0$ in $[0,t^*)$, we get $h'(t) < 0$. Secondly, as $h'$ is strictly
increasing and $h'(0) = -1$, we have $h'(t) > -1$ for $t \in
(0,t^*)$. This completes the proof.
\end{proof}

\section{Halley's Method Applied to the Majorizing Function}
\label{section:HalleyMethodAppliedToTheMajorizingFunction}

Let
\begin{equation}
\label{function:BanachSpaceIterativeFunctionHalleyMethod} H_F(x) :=
x - [\I - L_F(x)]^{-1}F'(x)^{-1}F(x)
\end{equation}
be the iterative function of Halley's method, where
$L_F(x)=\frac{1}{2}F'(x)^{-1}F''(x)F'(x)^{-1}F(x)$. Suppose $h$ is
the majorizing function to $F$ (see Definition
\ref{definition:MajorizingFunction}). Then Halley's method applied
to $h$ can be denoted as
\begin{equation}
\label{function:RealSpaceIterativeFunctionHalleyMethod} H_h(t) := t
- \frac{1}{1 - L_h(t)}\cdot\frac{h(t)}{h'(t)}, \ \ \ t \in [0,R),
\end{equation}
where $L_h(t) = h(t)h''(t)/(2h'(t)^2)$. In order to obtain the
convergence of the majorizing sequence generated by applying
Halley's method to the marjorizing function, we need some useful
lemmas.

\begin{lemma}
\label{lemma:estimate_Lh(t)} Let $h : [0,R) \to \mathbb{R}$ be a
twice continuously differentiable function and satisfy assumptions
$(A1)-(A3)$. Then we have $0 \leq L_h (t) \leq 1/4$ for $t \in
[0,t^*]$.
\end{lemma}

\begin{proof}
Define function
$$
\phi(s) = h(t) + h'(t)(s - t) + \frac{1}{2}h''(t)(s - t)^2,\ \ \ s
\in [t,t^*].
$$
Then, by Lemma \ref{lemma:MajorizingFunctionProperties} (ii),
$\phi(t) = h(t) > 0$ for $t \in [0,t^*)$. In addition, we have
\begin{equation}
\label{eq:Phi(t*)} \phi(t^*) = h(t) + h'(t)(t^* - t) +
\frac{1}{2}h''(t)(t^* - t)^2.
\end{equation}
By using Taylor's formula, one has that
\begin{equation}
\label{eq:h(t*)} h(t^*) = h(t) + h'(t)(t^* - t) +
\frac{1}{2}h''(t)(t^* - t)^2 + \int_0^1 (1 - \tau)[h''(t + \tau(t^*
- t)) - h''(t)](t^* - t)^2 \, \textup{d}\tau.
\end{equation}
In view of $h(t^*) = 0$ and $h''$ is increasing, it follows from
(\ref{eq:Phi(t*)}) and (\ref{eq:h(t*)}) that $ \phi(t^*) \leq 0$.
Thus, there exists a real root of $\phi(s)$ in $[t,t^*]$. So the
discriminant of $\phi(s)$ is greater than or equal to 0, i.e.,
$h'(t)^2 - 2h''(t)h(t) \geq 0$, which is equivalent to $0 \leq
h''(t)h(t)/h'(t)^2 \leq 1/2$. Therefore, $0 \leq L_h (t) \leq 1/4$
for $t \in [0,t^*]$. The proof is complete.
\end{proof}

\begin{lemma}
\label{lemma:estimate_Hh(t)} Let $h : [0,R) \to \mathbb{R}$ be a
twice continuously differentiable function and satisfy assumptions
$(A1)-(A3)$. Then, for all $t \in [0,t^*)$, $t < H_h(t) < t^*$.
Moreover, $h'(t^*) < 0$ if and only if there exists $t \in (t^*,R)$
such that $h(t) < 0$.
\end{lemma}

\begin{proof}
For $t \in [0,t^*)$, since $h(t) > 0$, $-1 < h'(t) < 0$ (from Lemma
\ref{lemma:MajorizingFunctionProperties}) and $0 \leq L_h (t) \leq
1/4$ (from Lemma \ref{lemma:estimate_Lh(t)}), one has that $t < H_h
(t)$. Furthermore, for any $t \in (0,t^*]$, it follows from the
definition of directional derivative
(\ref{definition:DirectionalDerivative}) and assumption (A2) that
$D^- h''(t)
> 0$. Thus, we have
\begin{equation*}
D^- H_h (t) = \frac{h(t)^2[3h''(t)^2 - 2 h'(t)D^-
h''(t)]}{[h(t)h''(t) - 2 h'(t)^2]^2} > 0,\ \ t \in (0,t^*].
\end{equation*}
This implies that $H_h(t) < H_h(t^*) = t^*$ for any $t \in (0,t^*)$.
So the first part of this Lemma is shown. For the second part, if
$h'(t^*) < 0$, then it is obvious that there exists $t \in (t^*,R)$
such that $h(t) < 0$. Conversely, noting that $h(t^*) = 0$, by Lemma
\ref{lemma:ConvexFunctionProperties1}, we have $h(t) > h(t^*) +
h'(t^*)(t - t^*)$ for $t \in (t^*,R)$, which implies $h'(t^*) < 0$.
This completes the proof.
\end{proof}

\begin{remark}
\label{remark:h'(t*)} The condition $h'(t^*) < 0$ in (A3) implies
the following properties:
\begin{enumerate}
\item[\textup{(a)}]
$h(t^{**}) = 0$ for some $t^{**} \in (t^*,R)$.
\item[\textup{(b)}]
$h(t) < 0$ for some $t \in (t^*,R)$.
\end{enumerate}
\end{remark}

In the usual versions of Kantorovich-type and Smale-type theorems
for Halley's method (e.g., \cite{HanWang1997,Han2001}), in order to
guarantee Q-cubic convergence, condition (a) is used. As we
discussed, this condition is more restrictive than condition
$h'(t^*) < 0 $ in assumption (A3).

\begin{lemma}
\label{lemma:estimate_t*-Hh(t)} Let $h : [0,R) \to \mathbb{R}$ be a
twice continuously differentiable function and satisfy assumptions
$(A1)-(A3)$. Then
\begin{equation}
\label{estimate:t*-Hh(t)} t^* - H_h(t) \leq \left[\frac{1}{3}
\frac{h''(t^*)^2}{h'(t^*)^2} + \frac{2}{9}
\frac{D^-h''(t^*)}{-h'(t^*)}\right](t^* - t)^3,\ \ \ t \in [0,t^*).
\end{equation}
\end{lemma}

\begin{proof}
By the definition of $H_h$ in
(\ref{function:RealSpaceIterativeFunctionHalleyMethod}), we may
derive the following relation
\begin{eqnarray*}
t^* - H_h(t) &=& \frac{1}{1 - L_h(t)} \left[(1 - L_h(t))(t^* - t) +
\frac{h(t)}{h'(t)}\right]\\
&=& - \frac{1}{h'(t)(1 - L_h(t))} \int_0^1 \big[h''(t + \tau(t^* -
t)) -
h''(t)\big](t^* - t)^2 (1 - \tau) \dif\tau\\
&& + \ \frac{t^* - t}{2((1 - L_h(t))} \frac{h''(t)}{h'(t)^2}\int_0^1
h''(t + \tau(t^* - t))(t^* - t)^2 (1 - \tau) \dif\tau.
\end{eqnarray*}
Since $h''$ is convex and $t < t^*$, it follows from Lemma
\ref{lemma:ConvexFunctionProperties2} (ii) that
$$
h''(t + \tau(t^* - t)) - h''(t) \leq [h''(t^*) -
h''(t)]\frac{\tau(t^* - t)}{t^* - t}.
$$
Then, noting that $h''$ is strictly increasing, we have
$$
t^* - H_h(t) \leq -\frac{h''(t^*) - h''(t)}{6h'(t)(1 - L_h(t))}(t^*
- t)^2 + \frac{h''(t^*)h''(t)}{4h'(t)^2(1 - L_h(t))}(t^* - t)^3.
$$
In view of the facts that $h'(t) < 0, h''(0) > 0$ and $h', h''$ are
strictly increasing on $[0,t^*)$ by Lemma
\ref{lemma:MajorizingFunctionProperties} and that $0 \leq L_h (t)
\leq 1/4$ for $t \in [0,t^*]$ by Lemma \ref{lemma:estimate_Lh(t)},
the preceding relation can be further reduced to
\begin{equation}
\label{ineq:estimatet*-Hh(t)} t^* - H_h(t) \leq
\frac{2}{9}\frac{h''(t^*) - h''(t)}{- h'(t)}(t^* - t)^2 +
\frac{1}{3} \frac{h''(t^*)^2}{h'(t^*)^2}(t^* - t)^3.
\end{equation}
As $h'$ is increasing, $h'(t^*) < 0$ and $h'(t) < 0$ in $[0,t^*)$,
we have
$$
\frac{h''(t^*) - h''(t)}{- h'(t)} \leq \frac{h''(t^*) - h''(t)}{-
h'(t^*)} = \frac{1}{- h'(t^*)}\frac{h''(t^*) - h''(t)}{t^* - t}(t^*
- t) \leq \frac{D^-h''(t^*)}{- h'(t^*)}(t^* - t),
$$
where the last inequality follows from Lemma
\ref{lemma:ConvexFunctionProperties2} (i). Combining the above
inequality with (\ref{ineq:estimatet*-Hh(t)}), we conclude that
(\ref{estimate:t*-Hh(t)}) holds. This completes the proof.
\end{proof}

By Definition \ref{definition:MajorizingFunction}, if $h$ is the
majorizing function to $F$ at $x_0$, then the results in Lemma
\ref{lemma:estimate_Lh(t)}, Lemma \ref{lemma:estimate_Hh(t)} and
Lemma \ref{lemma:estimate_t*-Hh(t)} also hold. Let $\{t_k\}$ denote
the majorizing sequence generated by
\begin{equation}
\label{majorizingsequence;tk} t_0 = 0,\ t_{k+1} = H_h(t_k) = t_k -
\frac{1}{1 - L_h(t_k)}\cdot\frac{h(t_k)}{h'(t_k)},\ \ \ k =
0,1,2,\ldots.
\end{equation}
Therefore, by using Lemma \ref{lemma:estimate_Hh(t)} and Lemma
\ref{lemma:estimate_t*-Hh(t)}, one concludes that

\begin{theorem}
\label{theorem:ConvergenceRealSpaceHalleyIterqtion} Let sequence
$\{t_k\}$ be defined by $(\ref{majorizingsequence;tk})$. Then
$\{t_k\}$ is well defined, strictly increasing and is contained in
$[0,t^*)$. Moreover, $\{t_k\}$ satisfies $(\ref{estimate:t*-Hh(t)})$
and converges to $t^*$ with Q-cubic.
\end{theorem}

\section{Semilocal Convergence Results for Halley's Method}
\label{section:SemilocalConvergenceOfHalleyMethod}

In this section, we study the semilocal convergence of Halley's
method (\ref{iteration:HalleyMethod}) in Banach space. Assume $F$ is
a twice differentiable nonlinear operator in some convex domain $D$.
For a given guess $x_0 \in D$, suppose that $F'(x_0)^{-1}$ exists.
The following lemmas, which provide clear relationships between the
majorizing function and the nonlinear operator, will play key roles
for the convergence analysis of Halley's method
(\ref{iteration:HalleyMethod}).

\begin{lemma}
\label{lemma:estimateF'(x)-1F'(x0)} Suppose $\|x - x_0\| \leq t <
t^*$. If $h : [0,t^*) \to \mathbb{R}$ is twice continuously
differentiable and is the majorizing function to $F$ at $x_0$. Then
$F'(x)$ is nonsingular and
\begin{equation}
\label{estimate;normF'(x)-1F'(x0)} \|F'(x)^{-1}F'(x_0)\| \leq -
\frac{1}{h'(\|x - x_0\|)} \leq - \frac{1}{h'(t)}.
\end{equation}
In particular, $F'$ is nonsingular in $\ball(x_0,t^*)$.
\end{lemma}

\begin{proof}
Take $x \in \overline{\ball(x_0,t)}$, $0 \leq t < t^*$. Since
\begin{equation*}
F'(x) = F'(x_0) + \int_0^1 [F''(x_0 + \tau(x - x_0)) - F''(x_0)](x -
x_0)\dif\tau + F''(x_0)(x - x_0),
\end{equation*}
by using conditions (\ref{condition:MajorantCondition}) and
(\ref{condition:InitialCondition}), we have
\begin{eqnarray*}
\|F'(x_0)^{-1}F'(x) - \I\| &\leq& \int_0^1
\|F'(x_0)^{-1}[F''(x_0^{\tau}) - F''(x_0)]\|\|x - x_0\|
d\tau + \|F'(x_0)^{-1}F''(x_0)\|\|x - x_0\|\\
&\leq& \int_0^1 \big[h''(\tau\|x - x_0\|) - h''(0)\big]\|x - x_0\|
d\tau + h''(0)\|x - x_0\|\\
&=& h'(\|x - x_0\|) - h'(0),
\end{eqnarray*}
where $x_0^{\tau} = x_0 + \tau(x - x_0)$. Since $h'(0) = - 1$ and
$-1 < h'(t) < 0$ for $(0,t^*)$ from Lemma
\ref{lemma:MajorizingFunctionProperties}, we get
$$
\|F'(x_0)^{-1}F'(x) - I\| \leq h'(t) - h'(0) < 1.
$$
It follows from Banach lemma that $F'(x_0)^{-1}F'(x)$ is nonsingular
and (\ref{estimate;normF'(x)-1F'(x0)}) holds. The proof is complete.
\end{proof}

\begin{lemma}
\label{lemma:estimateF'(x0)-1F''(x)} Suppose $\|x - x_0\| \leq t <
t^*$. If $h : [0,t^*) \to \mathbb{R}$ is twice continuously
differentiable and is the majorizing function to $F$ at $x_0$. Then
$\|F'(x_0)^{-1}F''(x)\| \leq h''(\|x - x_0\|) \leq h''(t)$.
\end{lemma}

\begin{proof}
By using (\ref{condition:MajorantCondition}), we have
\begin{eqnarray*}
\|F'(x_0)^{-1}F''(x)\| &\leq& \|F'(x_0)^{-1}[F''(x) -
F''(x_0)]\| + \|F'(x_0)^{-1}F''(x_0)\|\\
&\leq& h''(\|x - x_0\|) - h''(0) + h''(0) = h''(\|x - x_0\|).
\end{eqnarray*}
Since $h''$ is strictly increasing, we get $h''(\|x - x_0\|) \leq
h''(t)$. The proof is complete.
\end{proof}

\begin{lemma}
\label{lemma:ConvergenceAuxiliaryResults} Suppose that $h : [0,t^*)
\to \mathbb{R}$ is twice continuously differentiable. Let $\{x_k\}$
be generated by Halley's method $(\ref{iteration:HalleyMethod})$ and
$\{t_k\}$ be generated by $(\ref{majorizingsequence;tk})$. If $h$ is
the majorizing function to $F$ at $x_0$. Then, for all $k =
0,1,2,\ldots$, we have
\begin{enumerate}
\item[\textup{(i)}]
$F'(x_k)^{-1}$ exists and $\|F'(x_k)^{-1}F'(x_0)\| \leq - 1/h'(\|x_k
- x_0\|) \leq - 1/h'(t_k)$.
\item[\textup{(ii)}]
$\|F'(x_0)^{-1}F''(x_k)\| \leq h''(t_k)$.
\item[\textup{(iii)}]
$\|F'(x_0)^{-1}F(x_k)\| \leq h(t_k)$.
\item[\textup{(iv)}]
$[\I - L_F(x_k)]^{-1}$ exists and $\|[\I - L_F(x_k)]^{-1}\| \leq 1
/(1 - L_h(t_k))$.
\item[\textup{(v)}]
$\|x_{k+1} - x_k\| \leq t_{k+1} - t_k$.
\end{enumerate}
\end{lemma}

\begin{proof}
(i)-(v) are obvious for the case $k = 0$. Now we assume that they
hold for some $n \in \mathbb{N}$. By the inductive hypothesis (v)
and Theorem \ref{theorem:ConvergenceRealSpaceHalleyIterqtion}, we
have $\|x_{n+1} - x_0\| \leq t_{n+1} < t^*$. It follows from Lemma
\ref{lemma:estimateF'(x)-1F'(x0)} and Lemma
\ref{lemma:estimateF'(x0)-1F''(x)} that (i) and (ii) hold for $k =
n+1$, respectively. As for (iii), we can derive the following
relation from \cite{HanWang1997}:
\begin{eqnarray*}
F(x_{n+1}) &=& \frac{1}{2}F''(x_n)L_F(x_n)(x_{n+1} - x_n)^2 +
\int_0^1 (1 - \tau)[F''(x_n^{\tau}) - F''(x_n)](x_{n+1} - x_n)^2
\dif\tau,
\end{eqnarray*}
where $x_n^{\tau} = x_n + \tau(x_{n+1} - x_n)$. Applying
(\ref{condition:MajorantCondition}) and the inductive hypotheses
(i)-(ii) and (iv)-(v), we can obtain
\begin{eqnarray*}
\lefteqn{\|F'(x_0)^{-1}F(x_{n+1})\| \leq
\frac{1}{2}\|F'(x_0)^{-1}F''(x_n)\|\|L_F(x_n)\|\|x_{n+1} - x_n\|^2}\\
&& + \int_0^1 [h''(\tau\|x_{n+1} - x_n\| + \|x_n - x_0\|) -
h''(\|x_n - x_0\|)]\|x_{n+1} - x_n\|^2(1 - \tau) \dif\tau\\
&\leq& \frac{1}{4}h''(t_n)\frac{h(t_n)h''(t_n)}{h'(t_n)^2}(t_{n+1} -
t_n)^2 + \int_0^1 \big[h''(\tau(t_{n+1} - t_n) + t_n) -
h''(t_n)\big](t_{n+1} -
t_n)^2(1 - \tau) \dif\tau\\
&=& h(t_{n+1}).
\end{eqnarray*}
This means (iii) holds for $k = n +1$. By Lemma
\ref{lemma:estimate_Lh(t)} and the inductive hypotheses (i)-(iii),
we get (iv) for $k = n +1$. Finally, for (v), we have
\begin{eqnarray}
\label{estimate:normxn+2-xn+1} \|x_{n+2} - x_{n+1}\| &\leq& \|[\I -
L_F(x_{n+1})]^{-1}\| \|F'(x_{n+1})^{-1}F'(x_0)\|
\|F'(x_0)^{-1}F(x_{n+1})\|\nonumber\\
&\leq& - \frac{1}{1 - L_h(t_{n+1})}\frac{h(t_{n+1})}{h'(t_{n+1})} =
t_{n+2} - t_{n+1}.
\end{eqnarray}
Therefore, the statements hold for all $k = 0,1,2,\ldots$. This
completes the proof.
\end{proof}

We are now ready to prove the semilocal convergence results
(convergence, convergence rate and uniqueness) for Halley's method
(\ref{iteration:HalleyMethod}).

\begin{theorem}
\label{theorem:SemilocalConvergenceHalleyMethodMajorantCondition}
Let $F:D\subset X \to Y$ be a twice continuously differentiable
nonlinear operator, $D$ open and convex. Assume that there exists a
starting point $x_0\in D$ such that $F'(x_0)^{-1}$ exists, and that
$h$ is the majorizing function to $F$ at $x_0$, i.e.,
$(\ref{condition:MajorantCondition})$ and
$(\ref{condition:InitialCondition})$ hold and $h$ satisfies
assumptions $(A1)-(A3)$. Then the sequence $\{x_k\}$ generated by
Halley's method $(\ref{iteration:HalleyMethod})$ for solving
equation $(\ref{eq:NonlinearOperatorEquation})$ with starting point
$x_0$ is well defined, is contained in $\ball(x_0,t^*)$ and
converges to a point $x^* \in \overline{\ball(x_0,t^*)}$ which is
the solution of equation $(\ref{eq:NonlinearOperatorEquation})$.
\end{theorem}

\begin{proof}
By Lemma \ref{lemma:ConvergenceAuxiliaryResults}, we conclude that
the sequence $\{x_k\}$ is well defined. By Lemma
\ref{lemma:ConvergenceAuxiliaryResults} (v) and Theorem
\ref{theorem:ConvergenceRealSpaceHalleyIterqtion}, we have $\|x_k -
x_0\| \leq t_k < t^*$ for any $k \in \mathbb{N}$, which means that
$\{x_k\}$ is contained in $\ball(x_0,t^*)$. It follows from
(\ref{estimate:normxn+2-xn+1})  and Theorem
\ref{theorem:ConvergenceRealSpaceHalleyIterqtion} that
$$
\sum_{k = N}^\infty \|x_{k+1} - x_k\| \leq  \sum_{k = N}^\infty
(t_{k+1} - t_k) = t^* - t_N < + \infty,
$$
for any $N \in \mathbb{N}$. Hence $\{x_k\}$ is a Cauchy sequence in
$\ball(x_0,t^*)$ and so converges to some $x^* \in
\overline{\ball(x_0,t^*)}$. The above inequality also implies that
$\|x^* - x_k\| \leq t^* - t_k$ for any $k \in \mathbb{N}$. It
remains to prove that $F(x^*) = 0$. It follows from Lemma
\ref{lemma:estimateF'(x)-1F'(x0)} that $\{\|F'(x_k)\|\}$ is bounded.
By Lemma \ref{lemma:ConvergenceAuxiliaryResults}, we have
$$
\|F(x_k)\| \leq \|F'(x_k)\|\|F'(x_k)^{-1}F(x_k)\| \leq \|F'(x_k)\|(1
- L_h(t_k))(t_{k+1} - t_k).
$$
Letting $k \to \infty$, by noting the fact that $L_h(x_k)$ is
bounded (from Lemma \ref{lemma:estimate_Lh(t)}) and $\{t_k\}$ is
convergent, we have $\lim\limits_{k \to \infty} F(x_k) = 0$. Since
$F$ is continuous in $\overline{\ball(x_0,t^*)}$, $\{x_k\} \subset
\ball(x_0,t^*)$ and $\{x_k\}$ converges to $x^*$, we also have
$\lim\limits_{k \to \infty}F(x_k) = F(x^*)$. This completes the
proof.
\end{proof}

\begin{theorem}
\label{theorem:ConvergenceRate} Under the assumptions of Theorem
$\ref{theorem:SemilocalConvergenceHalleyMethodMajorantCondition}$,
we have the following error bound:
\begin{equation}
\label{estimate:ErrorEstimateHalleyMethod} \|x^* - x_{k+1}\| \leq
(t^* - t_{k+1})\left(\frac{\|x^* - x_k\|}{t^* - t_k}\right)^3,\ \ \
k = 0,1,\ldots.
\end{equation}
Thus, the sequence $\{x_k\}$ generated by Halley's method
$(\ref{iteration:HalleyMethod})$ converges Q-cubic as follows
\begin{equation}
\label{estimate:ErrorEstimateExactHalleyMethod} \|x^* - x_{k+1}\|
\leq \left[\frac{1}{3} \frac{h''(t^*)^2}{h'(t^*)^2} + \frac{2}{9}
\frac{D^-h''(t^*)}{-h'(t^*)}\right]\|x^* - x_k\|^3,\ \ \ k =
0,1,\ldots.
\end{equation}
\end{theorem}

\begin{proof}
Set $\Gamma_F = [I - L_F(x)]^{-1}$. Applying standard analytical
techniques, one has that
\begin{eqnarray*}
x^* - x_{k+1} &=& - \Gamma_F(x_k)F'(x_k)^{-1}[- F'(x_k)(x^* - x_k) -
F(x_k)] - \Gamma_F(x_k)L_F(x_k)(x^* - x_k)\\
&=& - \Gamma_F(x_k)F'(x_k)^{-1} \int_0^1 (1 - \tau)[F''(x_k^{\tau})
- F''(x_k)](x^* - x_k)^2 \dif\tau \\
&& + \frac{1}{2}\Gamma_F(x_k)F'(x_k)^{-1}F''(x_k)\left[F'(x_k)^{-1}
\int_0^1 (1 - \tau)F''(x_k^{\tau})(x^* - x_k)^2 \dif\tau\right](x^*
- x_k),
\end{eqnarray*}
where $x_k^{\tau} = x_k + \tau(x^* - x_k)$. Using
(\ref{condition:MajorantCondition}), one has that
$$
\int_0^1 \|F'(x_0)^{-1}[F''(x_k^{\tau}) - F''(x_k)]\|(1 - \tau)
\dif\tau \leq \int_0^1 [h''(\tau\|x^* - x_k\| + \|x_k - x_0\|) -
h''(\|x_k - x_0\|)](1 - \tau) \dif\tau.
$$
Then, we use Lemma \ref{lemma:ConvexFunctionProperties2} to obtain
\begin{eqnarray*}
h''(\tau\|x^* - x_k\| + \|x_k - x_0\|) - h''(\|x_k - x_0\|)
&\leq& h''(\tau\|x^* - x_k\| + t_k) - h''(t_k)\\
&\leq& [h''(\tau(t^* - t_k) + t_k) - h''(t_k)]\frac{\|x^* -
x_k\|}{t^* - t_k}.
\end{eqnarray*}
This together with Lemma \ref{lemma:estimateF'(x0)-1F''(x)} and
Lemma \ref{lemma:ConvergenceAuxiliaryResults}, we have
\begin{eqnarray*}
\|x^* - x_{k+1}\| &\leq& - \frac{1}{(1 - L_h(t_k))h'(t_k)}
\left[\int_0^1 [h''(\tau(t^* - t_k) + t_k) - h''(t_k)](1 - \tau)
\dif\tau\right]
\frac{\|x^* - x_k\|^3}{t^* - t_k}\\
&& + \frac{1}{2}\frac{h''(t_k)}{(1 - L_h(t_k))h'(t_k)^2}
\left[\int_0^1 h''(\tau(t^* - t_k) + t_k)(1 - \tau) \dif\tau\right]
\|x^* - x_k\|^3\\
&=& - \frac{1}{(1 - L_h(t_k))h'(t_k)}\left(\frac{\|x^* - x_k\|}{t^*
- t_k}\right)^3 h(t_k)\frac{t^* - t_{k+1}}{t_{k+1} - t_k} = (t^* -
t_{k+1})\left(\frac{\|x^* - x_k\|}{t^* - t_k}\right)^3.
\end{eqnarray*}
This shows (\ref{estimate:ErrorEstimateHalleyMethod}) holds for all
$k \in \mathbb{N}$. (\ref{estimate:ErrorEstimateExactHalleyMethod})
follows from Lemma \ref{lemma:estimate_t*-Hh(t)}. The proof is
complete.
\end{proof}

\begin{theorem}
\label{theorem:UniquenessSolution} Under the assumptions of Theorem
$\ref{theorem:SemilocalConvergenceHalleyMethodMajorantCondition}$,
the limit $x^*$ of the sequence $\{x_k\}$ is the unique zero of
equation $(\ref{eq:NonlinearOperatorEquation})$ in
$\ball(x_0,\rho)$, where $\rho$ is defined as $\rho := \sup\{t \in
[t^*,R): h(t) \leq 0\}$.
\end{theorem}

\begin{proof}
We first to show the solution $x^*$ of
(\ref{eq:NonlinearOperatorEquation}) is unique in
$\overline{\ball(x_0,t^*)}$. Assume that there exists another
solution $x^{**}$ in $\overline{\ball(x_0,t^*)}$. Then $\|x^{**} -
x_0\| \leq t^*$. Now we prove by induction that
\begin{equation}
\label{estimate:normx**-xk} \|x^{**} - x_k\| \leq t^* - t_k,\ \ \ k
= 0,1,2,\ldots.
\end{equation}
It is clear that the case $k = 0$ holds because of $t_0 = 0$. Assume
that the above inequality holds for some $n \in \mathbb{N}$. By
Theorem \ref{theorem:ConvergenceRate} we have
$$
\|x^{**} - x_{k+1}\| \leq (t^* - t_{k+1})\left(\frac{\|x^{**} -
x_k\|}{t^* - t_k}\right)^3.
$$
Then, by applying the inductive hypothesis
(\ref{estimate:normx**-xk}) to the above inequality, one has that
(\ref{estimate:normx**-xk}) also holds for $k = n+1$. Since
$\{x_k\}$ converges to $x^*$ and $\{t_k\}$ converges to $t^*$, from
(\ref{estimate:normx**-xk}) we conclude $x^{**} = x^*$. Therefore,
$x^*$ is the unique zero of (\ref{eq:NonlinearOperatorEquation}) in
$\overline{\ball(x_0,t^*)}$.

It remains to prove that $F$ does not have zeros in
$\ball(x_0,\rho)\backslash \overline{\ball(x_0,t^*)}$. For proving
this fact by contradiction, assume that $F$ does have a zero there,
i.e., there exists $x^{**} \in D \subset X$ such that $t^* <
\|x^{**} - x_0\| < \rho$ and $F(x^{**}) = 0$. We will show that the
above assumptions do not hold. Firstly, we have the following
observation;
\begin{equation}
\label{estimate:F(x**)} F(x^{**}) = F(x_0) + F'(x_0)(x^{**} - x_0) +
\frac{1}{2}F''(x_0)(x^{**} - x_0)^2 + (1 - \tau) \int_0^1
[F''(x_0^{\tau}) - F''(x_0)](x^{**} - x_0)^2 \dif\tau,
\end{equation}
where $x_0^{\tau} = x_0 + \tau(x^{**} - x_0)$. Secondly, we use
(\ref{condition:MajorantCondition}) to yield
\begin{eqnarray}
\lefteqn{\left\|(1 - \tau)\int_0^1 F'(x_0)^{-1}[F''(x_0^{\tau}) -
F''(x_0)](x^{**} - x_0)^2 \dif\tau\right\|}\nonumber\\
&\leq& \int_0^1 [h''(\tau\|x^{**} - x_0\|) - h''(0)]\|x^{**} -
x_0\|^2 (1 - \tau) \dif\tau \nonumber\\
&=& h(\|x^{**} - x_0\|) - h(0) - h'(0)\|x^{**} - x_0\| -
\frac{1}{2}h''(0) \|x^{**} - x_0\|^2. \label{estimate:F(x**)2}
\end{eqnarray}
Thirdly, by applying (\ref{condition:InitialCondition}), one has
that
\begin{eqnarray}
\lefteqn{\left\|F'(x_0)^{-1}[F(x_0) + F'(x_0)(x^{**} - x_0) +
\frac{1}{2}F''(x_0)(x^{**} - x_0)^2]\right\|}\nonumber\\
&\geq& \|x^{**} - x_0\| - \|F'(x_0)^{-1}F(x_0)\| - \frac{1}{2}
\|F'(x_0)^{-1}F''(x_0)\|\|x^{**} - x_0\|^2 \nonumber\\
&\geq& \|x^{**} - x_0\| - h(0) - \frac{1}{2}h''(0) \|x^{**} -
x_0\|^2. \label{estimate:F(x**)1}
\end{eqnarray}
In view of $F(x^{**}) = 0$ and $h'(0) = - 1$, combining
(\ref{estimate:F(x**)1}) and (\ref{estimate:F(x**)2}), we obtain
from (\ref{estimate:F(x**)}) that
$$
h(\|x^{**} - x_0\|) - h(0) + \|x^{**} - x_0\| - \frac{1}{2}h''(0)
\|x^{**} - x_0\|^2 \geq \|x^{**} - x_0\| - h(0) - \frac{1}{2}h''(0)
\|x^{**} - x_0\|^2,
$$
which is equivalent to $h(\|x^{**} - x_0\|) \geq 0$. Note that $h$
is strictly convex by Lemma
\ref{lemma:MajorizingFunctionProperties}. Hence $h$ is strictly
positive in the interval $(\|x^{**} - x_0\|,R)$. So, we get $\rho
\leq \|x^{**} - x_0\|$, which is a contradiction to the above
assumptions. Therefore, $F$ does not have zeros in
$\ball(x_0,\rho)\backslash \overline{\ball(x_0,t^*)}$ and $x^*$ is
the unique zero of equation (\ref{eq:NonlinearOperatorEquation}) in
$\ball(x_0,\rho)$. The proof is complete.
\end{proof}

\section{Special Cases}
\label{section:SpecialCases}

In this section we present two special cases of the convergence
results obtained in Section
\ref{section:SemilocalConvergenceOfHalleyMethod}. Namely,
convergence results under an affine covariant Lipschitz condition
and the $\gamma$-condition.

\subsection{Convergence results under the affine covariant Lipschitz condition}

In \cite{HanWang1997}, by using the majorizing technique, Han and
Wang studied the semilocal convergence of Halley's method
(\ref{iteration:HalleyMethod}) under affine covariant Lipschitz
condition:
\begin{equation}
\label{condition:LipschitzCondition} \|F'(x_0)^{-1}[F''(y) -
F''(x)]\| \leq L\|y - x\|,\ \ \ x,y \in D.
\end{equation}
The majorizing function employed in \cite{HanWang1997} is
\begin{equation}
\label{majorizingfunction:cubicfunction} f(t) = \beta - t +
\frac{\eta}{2} t^2 + \frac{L}{6} t^3.
\end{equation}

If we choose this cubic polynomial as the majorizing function $h$ in
(\ref{condition:MajorantCondition}), then we can see that the
majorant condition (\ref{condition:MajorantCondition}) reduced to
the Lipschitz condition (\ref{condition:LipschitzCondition}) and
that assumptions (A1) and (A2) are satisfied for $f$. Moreover, if
the following Kantorovich-type convergence criterion holds
\begin{equation}
\label{criterion:LipschitzConditionConvergenceCriterion} \beta < b
:= \frac{2(\eta + 2\sqrt{\eta^2 + 2L})}{3(\eta + \sqrt{\eta^2 +
2L})^2},
\end{equation}
then assumption (A3) is satisfied for $f$. Thus, the concrete forms
of Theorem
\ref{theorem:SemilocalConvergenceHalleyMethodMajorantCondition},
Theorem \ref{theorem:ConvergenceRate} and Theorem
\ref{theorem:UniquenessSolution} are given as follows.

\begin{theorem}
\label{theorem:SemilocalConvergenceHalleyMethodLipschitzCondition}
Let $F:D\subset X \to Y$ be a twice continuously differentiable
nonlinear operator, $D$ open and convex. Assume that there exists a
starting point $x_0\in D$ such that $F'(x_0)^{-1}$ exists, and
satisfies the affine covariant Lipschitz condition
$(\ref{condition:LipschitzCondition})$ and $\|F'(x_0)^{-1}F(x_0)\|
\leq \beta$, $\|F'(x_0)^{-1}F''(x_0)\| \leq \eta$. If
$(\ref{criterion:LipschitzConditionConvergenceCriterion})$ holds,
then the sequence $\{x_k\}$ generated by Halley's method
$(\ref{iteration:HalleyMethod})$ for solving equation
$(\ref{eq:NonlinearOperatorEquation})$ with starting point $x_0$ is
well defined, is contained in $\ball(x_0,t^*)$ and converges to a
point $x^* \in \overline{\ball(x_0,t^*)}$ which is the solution of
equation $(\ref{eq:NonlinearOperatorEquation})$, where $t^*$ is the
smallest positive root of $f$ $(defined \ by \
(\ref{majorizingfunction:cubicfunction}))$ in $[0,r_1]$, where $r_1
= 2/(\eta + \sqrt{\eta^2 + 2L})$ is the positive root of $f'$. The
limit $x^*$ of the sequence $\{x_k\}$ is the unique zero of equation
$(\ref{eq:NonlinearOperatorEquation})$ in $\ball(x_0,t^{**})$, where
$t^{**}$ is the root of $f$ in interval $[r_1,+\infty)$. Moreover,
the following error bound holds:
\begin{equation*}
\|x^* - x_{k+1}\| \leq (t^* - t_{k+1})\left(\frac{\|x^* - x_k\|}{t^*
- t_k}\right)^3,\ \ \ k = 0,1,\ldots.
\end{equation*}
And the sequence $\{x_k\}$ Q-cubically converges as follows:
\begin{equation*}
\|x^* - x_{k+1}\| \leq \frac{3(\eta + Lt^*)^2 + 2L(1 - \eta t^* -
Lt^{*2}/2)}{9(1 - \eta t^* - Lt^{*2}/2)^2}\|x^* - x_k\|^3,\ \ \ k =
0,1,\ldots.
\end{equation*}
\end{theorem}

%\begin{remark}
%Note that the Kantorovich-type convergence criterion
%(\ref{criterion:LipschitzConditionConvergenceCriterion}) in the
%above theorem differs slightly from the one given in
%\cite{HanWang1997}, i.e., $\beta \leq b$. Because $t^* =
%t^{**}$ when $\beta = b$, we have $f'(t^*) = 0$ which leads to that
%assumption (A3) is not satisfied. However, in \cite{HanWang1997},
%the Q-cubic convergence rate can be only obtained under the
%criterion $\beta < b$. In other words, the above convergence theorem
%meets the one presented in \cite{HanWang1997} for guaranteeing
%Q-cubic convergence rate.
%\end{remark}

\subsection{Convergence results under the $\gamma$-condition}

The notion of the $\gamma$-condition (see Definition
\ref{definition:GammaCondition}) for operators in Banach spaces was
introduced in \cite{WangHan1997} by Wang and Han to study Smale's
point estimate theory. In this subsection, we will give the
semilocal convergence results for Halley's method
(\ref{iteration:HalleyMethod}) under the $\gamma$-condition. As we
will discuss, these convergence results can be applied to Smale's
condition (see \cite{Smale1986} for more details about the Smale's
condition).

Smale \cite{Smale1986} studied the convergence and error estimation
of Newton's method (\ref{iteration:NewtonMethod}) under the
hypotheses that $F$ is analytic and satisfies
\begin{equation}
\label{condition:SmaleCondition}
\left\|F'(x_0)^{-1}F^{(n)}(x_0)\right\| \leq n!\gamma^{n-1},\ \ \ n
\geq 2,
\end{equation}
where $x_0$ is a given point in $D$ and $\gamma$ is defined by
\begin{equation}
\label{constant:GammaSmale} \gamma := \sup_{k >
1}\left\|\frac{F'(x_0)^{-1}F^{(k)}(x_0)}
{k!}\right\|^{\frac{1}{k-1}}.
\end{equation}
Wang and Han in \cite{WangHan1990} completely improved Smale's
results by introducing a majorizing function
\begin{equation}
\label{majorizingfunction:GammaCondition} f(t) = \beta - t +
\frac{\gamma t^2}{1 - \gamma t},\ \ \ \gamma > 0, \ 0 \leq t <
\frac{1}{\gamma}.
\end{equation}

If we choose this function as the majorizing function $h$, then we
can see that the majorant condition
(\ref{condition:MajorantCondition}) reduces to the following
condition:
\begin{equation}
\label{condition:SmaleMajorizingCondition} \|F'(x_0)^{-1}[F''(y) -
F''(x)]\| \leq \frac{2\gamma}{(1 - \gamma\|y - x\| - \gamma\|x -
x_0\|)^3} - \frac{2\gamma}{(1 - \gamma\|x - x_0\|)^3},\ \ \gamma >
0,
\end{equation}
where $\|y - x\| + \|x - x_0\| < 1/\gamma$, and that assumptions
(A1) and (A2) are satisfied for $f$. Moreover, if $\alpha :=
\beta\gamma < 3 - 2\sqrt{2}$, then assumption (A3) is satisfied for
$f$. Thus, the concrete forms of Theorem
\ref{theorem:SemilocalConvergenceHalleyMethodMajorantCondition},
Theorem \ref{theorem:ConvergenceRate} and Theorem
\ref{theorem:UniquenessSolution} are given as follows.

\begin{theorem}
\label{theorem:SemilocalConvergenceHalleyMethodSmaleMajorizingCondition}
Let $F:D\subset X \to Y$ be a twice continuously differentiable
nonlinear operator, $D$ open and convex. Assume that there exists a
starting point $x_0\in D$ such that $F'(x_0)^{-1}$ exists, and
satisfies condition $(\ref{condition:SmaleMajorizingCondition})$,
$\|F'(x_0)^{-1}F(x_0)\| \leq \beta$ and $\|F'(x_0)^{-1}F''(x_0)\|
\leq 2\gamma$. If $\alpha := \beta\gamma < 3 - 2\sqrt{2}$, then the
sequence $\{x_k\}$ generated by Halley's method
$(\ref{iteration:HalleyMethod})$ for solving equation
$(\ref{eq:NonlinearOperatorEquation})$ with starting point $x_0$ is
well defined, is contained in $\ball(x_0,t^*)$ and converges to a
point $x^* \in \overline{\ball(x_0,t^*)}$ which is the solution of
equation $(\ref{eq:NonlinearOperatorEquation})$. The limit $x^*$ of
the sequence $\{x_k\}$ is the unique zero of equation
$(\ref{eq:NonlinearOperatorEquation})$ in $\ball(x_0,t^{**})$, where
$t^*$ and $t^{**}$ are given as
\begin{equation}
\label{root:SmaleMajorizingFunctionfRoot} t^* = \frac{1 + \alpha -
\sqrt{(1 + \alpha)^2 - 8\alpha}}{4\gamma} \ \  \textup{and}\ \
t^{**} = \frac{1 + \alpha + \sqrt{(1 + \alpha)^2 -
8\alpha}}{4\gamma},
\end{equation}
respectively. Moreover, the following error bound holds:
\begin{equation}
\label{error:SmaleTypeSemilocalConvergenceErrorBound} \|x^* -
x_{k+1}\| \leq (t^* - t_{k+1})\left(\frac{\|x^* - x_k\|}{t^* -
t_k}\right)^3,\ \ \ k = 0,1,\ldots.
\end{equation}
And the sequence $\{x_k\}$ Q-cubically converges as follows
\begin{equation}
\label{rate:SmaleTypeSemilocalConvergenceRate} \|x^* - x_{k+1}\|
\leq \frac{8\gamma^2}{3[2(1 - \gamma t^*)^2 - 1]^2} \|x^* -
x_k\|^3,\ \ \ k = 0,1,\ldots.
\end{equation}
\end{theorem}

%\begin{remark}
%Note that the Smale-type convergence criterion in the above theorem
%is $\alpha := \beta\gamma < 3 - 2\sqrt{2}$, which differs from the
%usual version $\alpha := \beta\gamma \leq 3 - 2\sqrt{2}$ (see
%\cite{Smale1986,Wang1999,Han2001}). Since we have $t^* = t^{**}$
%when $\alpha = 3 - 2\sqrt{2}$, we have $f'(t^*) = 0$ and thus
%assumption (A3) is not satisfied. However, in \cite{Han2001}, the
%Q-cubic convergence rate of Halley's method
%(\ref{iteration:HalleyMethod}) can be only obtained under the
%criterion $\beta\gamma < 3 - 2\sqrt{2}$. In other words, for
%Halley's method (\ref{iteration:HalleyMethod}), the above
%convergence theorem meets the one presented in \cite{Han2001} for
%guaranteeing Q-cubic convergence rate.
%\end{remark}

The next result gives a condition easier to be checked than
condition (\ref{condition:MajorantCondition}), provided that the
majorizing function $h$ is thrice continuously differentiable.

\begin{lemma}
\label{lemma:RelationMajorantConditionAndGammaCondition2Order} Let
$F:D\subset X \to Y$ be a thrice continuously differentiable
nonlinear operator, $D$ open and convex. Let $h : [0,R) \to
\mathbb{R}$ be a thrice continuously differentiable function with
convex $h''$. Then $F$ satisfies condition
$(\ref{condition:MajorantCondition})$ if and only if
\begin{equation}
\label{condition:F'(x0)-1F'''(x)} \|F'(x_0)^{-1}F'''(x)\| \leq
h'''(\|x - x_0\|),
\end{equation}
for all $x \in D$ with $\|x - x_0\| < R$.
\end{lemma}

\begin{proof}
If $F$ satisfies (\ref{condition:MajorantCondition}), then
(\ref{condition:F'(x0)-1F'''(x)}) holds trivially. Conversely, if
$F$ satisfies (\ref{condition:F'(x0)-1F'''(x)}), then we have
\begin{eqnarray*}
\|F'(x_0)^{-1}[F''(y) - F''(x)]\| &\leq& \int_0^1
\|F'(x_0)^{-1}F'''(x +
\tau (y - x))\| \|y - x\| \dif\tau\\
&\leq& \int_0^1 h'''(\|x - x_0\| + \tau\|y - x\|)\|y - x\|
\dif\tau\\
&=& h''(\|y - x\| + \|x - x_0\|) - h''(\|x - x_0\|),
\end{eqnarray*}
which implies that $F$ satisfies
(\ref{condition:MajorantCondition}). The proof is complete.
\end{proof}

If the majorizing function $h$ is defined by
(\ref{majorizingfunction:GammaCondition}), then
(\ref{condition:F'(x0)-1F'''(x)}) becomes
\begin{equation}
\label{condition:GammaCondition2Order} \|F'(x_0)^{-1}F'''(x)\| \leq
\frac{6\gamma^2}{(1 - \gamma\|x - x_0\|)^4},
\end{equation}
which means that $F$ satisfies the $\gamma$-condition with 2-order
(see Definition \ref{definition:GammaCondition}) in $\ball(x_0,R)$.
By \cite{WangHan1997}, if $F$ satisfies the $\gamma$-condition with
2-order, then $F$ satisfies the $\gamma$-condition (with 1-order).

One typical and important class of examples satisfying the
$\gamma$-condition with 2-order
(\ref{condition:GammaCondition2Order}) is the one of analytic
functions. The following lemma shows that an analytic operator
satisfies the $\gamma$-condition with 2-order.

\begin{lemma}
\label{lemma:RelationGammaConditionAndAnalyticOperator} Let $F : D
\to Y$ be an analytic nonlinear operator. Suppose that $x_0 \in D$
is a given point, $F'(x_0)$ is invertible and that
$\ball(x_0,1/\gamma) \subset D$. Then $F$ satisfies the
$\gamma$-condition with 2-order
$(\ref{condition:GammaCondition2Order})$ in $\ball(x_0,1/\gamma)$,
where $\gamma$ is defined by $(\ref{constant:GammaSmale})$.
\end{lemma}

\begin{proof}
For any $x \in D$, since $F$ is an analytic operator, we have
$$
F'(x_0)^{-1}F'''(x) = \sum_{n=0}^\infty \frac{1}{n!}
F'(x_0)^{-1}F^{(n + 3)}(x_0)(x - x_0)^n.
$$
This together with (\ref{constant:GammaSmale}) directly leads to
$$
\|F'(x_0)^{-1}F'''(x)\| \leq \gamma^2 \sum_{n=0}^\infty
(n+3)(n+2)(n+1)(\gamma\|x - x_0\|)^n.
$$
Noting that $\gamma\|x - x_0\| < 1$ due to the assumption
$\ball(x_0,1/\gamma) \subset D$, we have
$$
\sum_{n=0}^\infty (n+3)(n+2)(n+1)(\gamma\|x - x_0\|)^n =
\frac{6}{(1-\gamma\|x - x_0\|)^4},
$$
which completes the proof.
\end{proof}

From Lemma
\ref{lemma:RelationMajorantConditionAndGammaCondition2Order} and
Lemma \ref{lemma:RelationGammaConditionAndAnalyticOperator}, we
conclude that the semilocal convergence results obtained in Theorem
\ref{theorem:SemilocalConvergenceHalleyMethodSmaleMajorizingCondition}
also hold when $F$ is an analytic operator.

\begin{theorem}
\label{theorem:SemilocalConvergenceForAnalyticOperator} Let $F: D
\to F$ be an analytic operator, $D$ open and convex. Assume that
exists $x_0 \in D$ such that $F'(x_0)$ is nonsingular. If
$\|F'(x_0)^{-1}F(x_0)\| \leq \beta$ and $\alpha := \beta\gamma < 3 -
2\sqrt{2}$, where $\gamma$ is given by
$(\ref{constant:GammaSmale})$. Then the sequence $\{x_k\}$ generated
by Halley's method $(\ref{iteration:HalleyMethod})$ for solving
equation $(\ref{eq:NonlinearOperatorEquation})$ with starting point
$x_0$ is well defined, is contained in $\ball(x_0,t^*)$ and
converges to a point $x^* \in \overline{\ball(x_0,t^*)}$ which is
the solution of equation $(\ref{eq:NonlinearOperatorEquation})$. The
limit $x^*$ of $\{x_k\}$ is the unique zero of equation
$(\ref{eq:NonlinearOperatorEquation})$ in $\ball(x_0,t^{**})$, where
$t^*$ and $t^{**}$ are given in
$(\ref{root:SmaleMajorizingFunctionfRoot})$. Moreover, the error
estimate and the convergence rate for $\{x_k\}$ are characterized by
$(\ref{error:SmaleTypeSemilocalConvergenceErrorBound})$ and
$(\ref{rate:SmaleTypeSemilocalConvergenceRate})$, respectively.
\end{theorem}

\section{Remarks and Numerical Example}

All the well-known one-point iterative methods with third-order of
convergence are given by the following unified form (see
\cite{Hernandez2005,Hernandez2009} for more details):
\begin{equation}
\label{iteration:IterativeFamilyMethods} \left\{
\begin{array}{l}
x_{n + 1} = x_n - H(L_F(x_n)) F'(x_n)^{-1} F(x_n),\\
H(L_F(x_n)) = \I + \frac{1}{2} L_F(x_n) + \sum_{k \geq 2} a_k
L_F(x_n)^k,\\
L_F(x_n) = F'(x_n)^{-1} F''(x_n)F'(x_n)^{-1}F(x_n), \ \ \ n \in
\mathbb{N},
\end{array} \right.
\end{equation}
where $\{a_k\}_{k \geq 2}$ is a nonnegative and nonincreasing real
sequence such that
$$
\sum_{k = 0}^\infty a_k t^k < + \infty, \ \  t \in [- \frac{1}{2},
\frac{1}{2}] \ \ \ \text{with} \ \ a_0 = 1, a_1 = \frac{1}{2}.
$$
Thus, if $L_F(x_n)$ exists and $\|L_F(x_n)\| \leq 1/2$, then
(\ref{iteration:IterativeFamilyMethods}) is well defined. In
particular, when $a_k = 1/2^k$ for any $k \geq 0$,
(\ref{iteration:IterativeFamilyMethods}) reduces to Halley's method
(\ref{iteration:HalleyMethod}).

Hern\'{a}ndez and Romero in \cite{Hernandez2009} studied the
semilocal convergence of (\ref{iteration:IterativeFamilyMethods})
under the following condition:
\begin{equation}
\label{condition:MajorantConditionLike} \|F''(x) - F''(y)\| \leq
|p''(u) - p''(v)|, \ \ x, y \in D, u, v \in [a, s] \ \text{such
that} \ \|x - y\| \leq |u - v|,
\end{equation}
where $p$ is a sufficiently differentiable nonincreasing and convex
real function in an interval $[a, b]$ such that $p(a) > 0 > p(b)$
and $p'''(t) \geq 0$ in $[a, s]$, and $s$ is the unique simple
solution of $p(t) = 0$ in $[a, b]$.

We point out that condition (\ref{condition:MajorantCondition}) used
in our convergence analysis is affine invariant but not condition
(\ref{condition:MajorantConditionLike}) (see \cite{Deuflhard1979,
Deuflhard2004} for more details about the affine invariant theory),
and that the assumptions of the majorizing function used in our
analysis are weaker than the ones in \cite{Hernandez2009}.
Furthermore, our convergence analysis provides a clear relationship
between the majorizing function and the nonlinear operator, see
Lemmas \ref{lemma:estimateF'(x)-1F'(x0)},
\ref{lemma:estimateF'(x0)-1F''(x)} and
\ref{lemma:ConvergenceAuxiliaryResults}.

To illustrate the theoretical results, we provide a numerical
example on nonlinear Hammerstein integral equation of the second
kind. Consider the integral equation:
\begin{equation}
\label{eq:NonlinearHammersteinEquation} u(s) = f(s) + \lambda
\int_a^b k(s, t) u(t)^n \dif t, \ \ \lambda \in \mathbb{R}, n \in
\mathbb{N},
\end{equation}
where $f$ is a given continuous function satisfying $f(s) > 0$ for
$s \in [a, b]$ and the kernel function $k(s, t)$ is continuous and
positive in $[a, b] \times [a, b]$. Let $X = Y = C[a, b]$ and $D =
\{u \in D[a, b]: u(s) \geq 0, s \in [a, b]\}$. Then, finding a
solution of (\ref{eq:NonlinearHammersteinEquation}) is equivalent to
find a solution of $F(x) = 0$, where $F: D \to C[a, b]$ is defined
by
$$
F(u)(s) = u(s) - f(s) - \lambda \int_a^b k(s, t) u(t)^n \dif t, \ \
s \in [a, b], \lambda \in \mathbb{R}, n \in \mathbb{N}.
$$
We adopt the max-norm. The first and second derivative of $F$ are
given by
$$
F'(u)v(s) = v(s) - n \lambda \int_a^b k(s, t) u(t)^{n - 1} v(t) \dif
t, \ \ v \in D,
$$
and
$$
F''(u)[vw](s) = - n(n - 1)\lambda \int_a^b k(s, t)u(t)^{n - 2}
(vw)(t) \dif t, \ \ v, w \in D.
$$

We choose $[a, b] = [0, 1], n = 3, x_0(t) = f(t) = 1$ and $k(s, t)$
as the Green kernel on $[0,1] \times [0, 1]$ defined by
$$
G(s,t)=
\begin{cases}
\displaystyle \frac{(b - s)(t - a)}{b - a} = t(1 - s),\ t \leq s,\\
\displaystyle \frac{(b - t)(s - a)}{b - a} = s(1 - t),\ s \leq t.
\end{cases}
$$
Let $M = \max\limits_{s \in [0, 1]} \int_0^1 |k(s, t)| \dif t$. Then
$M = 1/8$. Thus, we obtain that
$$
\|F'(x_0)^{-1}\| \leq \frac{8}{8 - 3 |\lambda|}, \ \
\|F'(x_0)^{-1}F(x_0)\| \leq \frac{|\lambda|}{8 - 3 |\lambda|}, \ \
\|F'(x_0)^{-1}F''(x_0)\| \leq \frac{6 |\lambda|}{8 - 3 |\lambda|}.
$$
In addition, for any $x, y \in D$, we have
$$
\|F'(x_0)^{-1}[F''(x) - F''(y)]\| \leq \frac{6 |\lambda|}{8 - 3
|\lambda|} \|x - y\|.
$$
So, we obtain the values of $\beta, \eta$ and $L$ in
(\ref{majorizingfunction:cubicfunction}) as follows:
$$
\beta = \frac{|\lambda|}{8 - 3 |\lambda|}, \ \ \eta = \frac{6
|\lambda|}{8 - 3 |\lambda|}, \ \ L = \frac{6 |\lambda|}{8 - 3
|\lambda|}.
$$
Consequently, the convergence criterion
(\ref{criterion:LipschitzConditionConvergenceCriterion}) holds for
any $|\lambda| \in [0, 32/27)$, and Theorem
\ref{theorem:SemilocalConvergenceHalleyMethodLipschitzCondition} is
applicable and the sequence generated by Halley's method
(\ref{iteration:HalleyMethod}) with initial point $x_0$ converges to
a zero of $F$ defined by (\ref{eq:NonlinearHammersteinEquation}).

For the special cases of integral equation
(\ref{eq:NonlinearHammersteinEquation}) with $n = 3$ when $\lambda =
1/4, 1/2, 3/4, 1$ and $f(t) = 1$, the corresponding domains of
existence and uniqueness of solution, together with those obtained
by Hern\'{a}ndez and Romero in \cite{Hernandez2007}, are given in
Table \ref{table:DomainExistenceUniqueness}. We notice that our
convergence analysis gives better existence balls and uniqueness
fields than those in \cite{Hernandez2007}.

\begin{table}
\centering \caption{Domains of existence and uniqueness of solution
for Halley's method} \label{table:DomainExistenceUniqueness}
\begin{tabular}{lllll}
\hline
$f(t) = 1$ & & & \multicolumn{2}{l}{Hern\'{a}ndez and Romero \cite{Hernandez2007}}\\
\cline{1-1} \cline{4-5} $\lambda$ & Existence & Uniqueness &
Existence & Uniqueness \\
\hline
0.25 & $\overline{\ball(1, 0.0346081)}$ & \ball(1, 4.06814) &
$\overline{\ball(1, 0.0348595)}$ & \ball(1, 4.06798) \\
0.5 & $\overline{\ball(1, 0.0783777)}$ & \ball(1, 2.35026) &
$\overline{\ball(1, 0.0814400)}$ & \ball(1, 2.34809) \\
0.75 & $\overline{\ball(1, 0.138260)}$ & \ball(1, 1.54454) &
$\overline{\ball(1, 0.157580)}$ & \ball(1, 1.52953) \\
1 & $\overline{\ball(1, 0.236068)}$ & \ball(1, 1) &
$\overline{\ball(1, 0.402436)}$ & \ball(1, 0.853166)\\
\hline
\end{tabular}
\end{table}

%\noindent\textbf{\\Acknowledgement\\}
%
%The authors thank the anonymous referees for their valuable comments
%which have considerably improved this paper.

%\bibliography{LingBibTexData}
%\bibliographystyle{plain}

\end{document}